\documentclass[10pt,a4paper]{article}
\usepackage[utf8]{inputenc}
\usepackage[english]{babel}
\usepackage{amsmath}
\usepackage{amsfonts}
\usepackage{amssymb}
\usepackage{hyperref}
\usepackage{mathrsfs}

\usepackage{amsthm}
\newtheorem{theorem}{Theorem}[section]
\newtheorem{proposition}[theorem]{Proposition}
\newtheorem{corollary}[theorem]{Corollary}
\newtheorem{lemma}[theorem]{Lemma}
\newtheorem{remark}[theorem]{Remark}

\newenvironment{definition}[1][Definition]{\begin{trivlist}
\item[\hskip \labelsep {\bfseries #1}]}{\end{trivlist}}
\newenvironment{example}[1][Example]{\begin{trivlist}
\item[\hskip \labelsep {\bfseries #1}]}{\end{trivlist}}

\usepackage[babel]{csquotes}
\usepackage[backend=bibtex,style=alphabetic]{biblatex}
\DeclareFieldFormat[article]{volume}{Vol. \addnbspace #1\addnbspace}
\DeclareFieldFormat[article]{number}{No.\addnbspace #1\addnbspace}

\usepackage{pgf,tikz}
\usetikzlibrary{positioning,shapes,decorations.markings}
\usetikzlibrary{arrows}
\usetikzlibrary[patterns]
\usetikzlibrary{cd}

\usepackage[sc]{mathpazo} % Use the Palatino font
\usepackage{microtype} % Slightly tweak font spacing for aesthetics

\usepackage{titlesec} % Allows customization of titles
\titleformat{\section}[block]{\large\scshape\centering}{\thesection.}{1em}{} % Change the look of the section titles
\titleformat{\subsection}[runin]
  {\normalfont\large\bfseries}{\thesubsection}{1em}{}
\titleformat{\subsubsection}[runin]
  {\normalfont\normalsize\bfseries}{\thesubsubsection}{1em}{}

\begin{document}
\title{Julia sets of complex H\'enon maps}
\author{Lorenzo Guerini and Han Peters}
\maketitle

\begin{abstract}
There are two natural definitions of the Julia set for complex H\'enon maps: the sets $J$ and $J^\star$. Whether these two sets are always equal is one of the main open questions in the field. We prove equality when the map acts hyperbolically on the a priori smaller set $J^\star$, under the additional hypothesis of substantial dissipativity. This result was claimed, without using the additional assumption, in the paper \cite{F}, but the proof is incomplete. Our proof closely follows ideas from \cite{F}, deviating at two points where substantial dissipativity is used.

We show that $J = J^\star$ also holds when hyperbolicity is replaced by one of two weaker conditions. The first is quasi-hyperbolicity, introduced in \cite{BS8}, a natural generalization of the one dimensional notion of semi-hyperbolicity. The second is the existence of a dominated splitting on $J^\star$. Substantially dissipative H\'enon maps admitting a dominated splitting on the possibly larger set $J$ were recently studied in in \cite{LP}.
\end{abstract}

\section{Introduction}

The Julia set plays an central role in the study of one dimensional holomorphic dynamical systems. There are several natural analogies for the one-dimensional Julia set when one studies the iteration of complex H\'enon maps. When only considering the forward dynamics, the natural definition is the set $J^+$, the boundary of $K^+$, the set of points with bounded orbits. Note that $J^+$ is also the set where the sequence of forward iterates locally does not form a normal family, and the set where the pluri-complex Green's function $G^+$ is not pluri-harmonic. The fact that these three definitions all lead to the same set is in complete analogy with the one-dimensional setting. Moreover, $J^+$ is equal to the support of $\mu^+ := dd^c G^+$, which in contrast with the one-dimension setting is not a measure but a $(1,1)$-current.

Considering the same objects for the backwards iterations of the invertible H\'enon maps leads to definitions of $J^-$, $G^-$ and $\mu^-$.

As was shown in \cite{BS1}, the wedge product $\mu = \mu^+ \wedge \mu^-$ can be defined and gives the unique measure of maximal entropy $\log(d)$. The support of this measure is denoted by $J^\star$. The set $J^\star$ is the closure of the set of periodic saddle points, see \cite{BLS}. There are therefore good reasons to consider the set $J^\star$ as the Julia set of the invertible dynamical system.

There is however another natural definition of the Julia set, namely the set $J = J^+ \cap J^-$. Since $J^+$ and $J^-$ are the respective supports of the currents $\mu^+$ and $\mu^-$, it follows immediately that $J^\star \subset J$. However, it is not known whether the opposite inclusion $J \subset J^\star$ also holds. The possible equality of the two sets $J$ and $J^\star$ is one of the most important open questions regarding the dynamics of complex H\'enon maps.

It is often more natural to make assumptions regarding the dynamics on the smaller Julia set $J^\star$, but one the other hand it is usually more  convenient to have these assumptions on the larger set $J$. In \cite{BS1} a complete description of uniformly hyperbolic H\'enon maps was given, and in particular it was shown that for those maps $J$ equals $J^\star$. In \cite{BS1} a H\'enon map is defined to be uniformly hyperbolic if it acts hyperbolically on the invariant set $J$. However, considering that $J^\star$ is the closure of the set of saddle points, the assumption that $f$ acts uniformly hyperbolically on $J^\star$ may be more natural, and can be tested using only the derivatives at those saddle points. A natural question therefore is whether the assumption that a H\'enon map acts uniformly hyperbolically on $J^\star$ is sufficient to conclude that $J = J^\star$, and in particular that the map also acts hyperbolically on $J$.

The paper \cite{F} claimed to have proved this statement, but unfortunately the proof is incomplete. We do not know how to complete the proof in general, but we will show here that the proof can be corrected under an additional assumption on the Jacobian determinant, namely:
$$
|\mathrm{Jac}(f)| < \frac{1}{\mathrm{deg}(f)^2}.
$$
We will refer to this condition as \emph{substantial dissipativity}. This condition was recently applied in \cite{DL}, \cite{LP} and \cite{LP2}, in each paper in order to deduce that stable manifolds are ``unbridged'', a notion originally introduced by Teisuke Jin. We will discuss substantial dissipativity and unbridged stable manifolds in more detail in section \ref{section:unbridged}.

Our main result is the following.

\begin{theorem}
\label{maintheorem}
Let $f$ be a substantial dissipative H\'enon map, and assume that $f$ is uniformly hyperbolic on $J^\star$. Then $J = J^\star$.
\end{theorem}

It turns out that the hyperbolicity assumption can be significantly weakened. Instead of requiring that the map is uniformly hyperbolic on $J^\star$, the proof merely requires that $f$ is \emph{quasi-contracting}, see \cite{BS8}. As a consequence we obtain the following two results.

\begin{corollary}
Let $f$ be a substantial dissipative H\'enon map, and assume that $f$ is either quasi-hyperbolic, or admits a dominated splitting on $J^\star$. Then $J = J^\star$.
\end{corollary}

We note that in the dominated splitting setting the conclusion $J = J^\star$ was also obtained in \cite{LP2}, using completely different methods, under the stronger assumption that $f$ admits a dominated splitting on $J$. In \cite{D} the equality $J=J^\star$ was obtained under the assumption that the current $\mu^-$ has \emph{no degree growth} on a large bidisk. There are clear analogies between our approach and the proof given in \cite{D}, see Remark \ref{dujardin} for more details.

The notion of quasi-expansion was introduced in \cite{BS8}, and is strictly weaker than the assumption that $f$ is uniformly hyperbolic on $J^\star$. We will review the definition and properties of quasi-expanding H\'enon maps in sections \ref{section:quasi} and \ref{section:maximal}.

The proof of our main result will be completed in section \ref{section:proof}. We closely follow the approach from \cite{F}. The fact that stable manifold are unbridged plays an important role in the description of $J^+$ given in Lemma \ref{lemma:5.5}. As a consequence if $q \in J \setminus J^\star$, then $q \in W^s(p)$ for some point $p \in J^\star$. It follows then that $q$ is contained in a subset $U \subset W^s(q_0)$, open in the intrinsic topology, where $G^- \equiv 0$. We can conclude that there also exists a point $q_1 \in J^\star$ in the closure $\overline{U}\subset W^s(q_0)$. By iterating backwards and passing to a normal limit of the stable manifolds $W^s(f^{-n}(q_1))$ we obtain a contradiction with the fact that stable manifolds are unbridged.

We first discuss background on quasi-hyperbolicity in section \ref{section:quasi}. In section \ref{section:maximal} we continue the study of these maps, introducing a maximal normal family of parametrizations of stable manifolds. In section \ref{section:unbridged} we show that the substantial dissipativity assumption implies that these stable manifolds are all unbridged.

\medskip

\noindent {\bf Acknowledgment.} The authors would like to thank Eric Bedford for the inspiring discussions that stimulated this research.

\section{Quasi-hyperbolic H\'enon maps}\label{section:quasi}

\noindent{\bf H\'enon maps} We will refer to a H\'enon map $f$ as a finite composition of maps of the form
$$
(x,y) \rightarrow (p_i(x) - \delta_i y, x),
$$
where each $p_i$ is a polynomial of degree at least $2$, and each $\delta_i \in \mathbb C\setminus\{0\}$. The algebraic degree of $f$, denoted by $d$, is the product of the degrees of the polynomials $p_i$.

Complex H\'enon maps were first studied in \cite{Hu}. It was shown in \cite{FM} that every polynomial automorphism of $\mathbb C^2$ with non-trivial dynamics is conjugate to a composition of maps as above. For more background on the dynamics of H\'enon maps we refer the reader to \cite{BS1} and later papers in the series. We recall the notation
$$
\begin{aligned}
K^+ &= \{z \in \mathbb C^2 : \{f^n(z)\}_{n \in \mathbb N}  \text{ is bounded }\},\\
J^+ &= \partial K^+,
\end{aligned}
$$
and $K^-$ and $J^-$ similarly.

A rough but useful description of the global dynamics is provided by the \emph{filtration}. For $R>0$ define the sets
$$
\begin{aligned}
V^+ & = \{(x,y) \mid |x| \ge \max(|y|, R)\}\\
V^- & = \{(x,y) \mid |y| \ge \max(|x|, R)\}, \; \; \mathrm{and}\\
\Delta^2_R & = \{(x,y) \mid \max(|x|,|y|) < R\}.
\end{aligned}
$$
For $R$ sufficiently large it follows that $f(V^+) \subset V^+$ and all forward orbits in $V^+$ converge to $[1:0:0] \in \ell^\infty$, while $f^{-1}(V^-) \subset V^-$ and all backward orbits in $V^-$ converge to $[0:1:0]$. An immediate consequence of the filtration structure is that the set $K=K^+\cap K^-$ is bounded.

The notion of quasi-hyperbolicity was introduced in \cite{BS8} as a natural generalization of semi-hyperbolicity to the two-dimensional setting.  Several equivalent definitions of quasi-hyperbolicity were given. The definition we will adopt uses the pluri-complex Green's functions $G^+$ and $G^-$. We recall their definitions and some of their properties.

\begin{definition}
Let $f:\mathbb C^2\rightarrow \mathbb C^2$ be a H\'enon map. We define the \textit{pluri-complex Green's functions} $G^{+/-}$ as
\[
G^{+/-}(z):=\lim_{n\to\infty} \frac{\log^+ \Vert f^{+n/-n}(z)\Vert}{d^n},
\]
where $\log^+(t)=\max\{\log(t),0\}$ for $t\geq 0$.
\end{definition}

The functional equations
\[
G^+(f(z))=dG^+(z),\qquad G^-(f^{-1}(z))=dG^-(z)
\]
follow directly from the definitions.

\begin{theorem}
The function $G^+$ is characterized by the properties
\begin{itemize}
\item $G^+$ is continuous on $\mathbb C^2$, $G^+ \equiv 0$ on $K^+$ and $G^+>0$ on $\mathbb C^2\setminus K^+$;
\item $G^+$ is pluri-subharmonic on $\mathbb C^2$. In addition to this, $G^+$ is pluri-harmonic on $\mathbb C^2\setminus K^+$;
\item $G^+(z)\leq \log\Vert z\Vert+O(1)$ and $\limsup_{z\to\infty}G^+(z)/\log\Vert z\Vert=1$.
\end{itemize}
Analogous properties hold for $G^-$.
\end{theorem}

We define the $(1,1)$-currents $\mu^+$ and $\mu^-$ by
\begin{equation}
\label{currents}
\mu^+:=dd^c G^+,\qquad \mu^-:=dd^c G^-.
\end{equation}
One has $supp(\mu^+)=J^+$ and $supp(\mu^-)=J^-$. The wedge product $\mu:=\mu^+\wedge\mu^-$ defines an $f$-invariant probability measure, called the equilibrium measure. We let
\[
J^\star:=supp(\mu).
\]
It is clear that $J^\star\subset J$, but it is not known whether the two sets are equal for all H\'enon maps. Corollary 9.3 of \cite{BLS} describes the set $J^\star$ in terms of the saddle points of $f$, i.e. the hyperbolic periodic points of $f$.
\begin{theorem}
\label{closuresaddle}
$J^\star$ is the closure of the set of saddle points.
\end{theorem}

\emph{Quasi-expansion} is defined in terms of the existence of a normal family of parameterizations contained in $J^-$. Let $\mathcal S\subset J^\star$ be a dense, $f$-invariant set and suppose that for every $p\in\mathcal S$ there exists a injective immersion $\xi_p:\mathbb C\rightarrow \mathbb C^2$ such that
\begin{equation}
\label{p1}
p\in\xi_p(\mathbb C),\qquad \xi_p(\mathbb C)\subset J^-.
\end{equation}
In addition suppose that
\begin{equation}
\label{p2}
f(\xi_p(\mathbb C))=\xi_{f(p)}(\mathbb C),
\end{equation}
and that, for every $p_1$ and $p_2$ in $\mathcal S$, $\xi_{p_1}(\mathbb C)$ and $\xi_{p_2}(\mathbb C)$ are either disjoint or they coincide, i.e.
\begin{equation}
\label{p3}
\xi_{p_1}(\mathbb C)\cap\xi_{p_2}(\mathbb C)\neq \emptyset\Rightarrow\xi_{p_1}(\mathbb C)=\xi_{p_2}(\mathbb C).
\end{equation}
For every $p\in\mathcal S$, we write $\Psi_p$ for the set of all the maps of the form $\psi_p(\zeta)=\xi_p(a\zeta+b)$, satisfying the normalization properties
\begin{equation}
\label{normalizationprop}
\psi_p(0)=p,\qquad \max_{|\zeta|\leq 1}G^+\circ \psi_p(\zeta)=1.
\end{equation}
Two distinct elements of $\Psi_p$ are equal up to a rotation. We further let $\Psi_{\mathcal S}=\bigcup_{p\in\mathcal S}\Psi_p$.

\begin{definition}
A H\'enon map $f$ is \textit{quasi-expanding} if there exists $\mathcal S\subset J^\star$ and $\Psi_\mathcal S$ ad above, such that $\Psi_\mathcal S$ is normal. A map $f$ is \textit{quasi-contracting} if $f^{-1}$ is quasi-expanding, and \textit{quasi-hyperbolic} if it is both quasi-expanding and quasi-contracting.
\end{definition}

We recall two natural choices of a $f$-invariant set $\mathcal S$, dense in $J^\star$, such that at every point $p\in\mathcal S$ there exists an injective map satisfying  \eqref{p1}, \eqref{p2} and \eqref{p3}.

\begin{example}[Saddle Points]
We write $\mathcal S_1$ for the set of the saddle points. $\mathcal S_1\subset J^\star$ is  a dense and $f$-invariant set. For every $p\in\mathcal S_1$, by the (Un)Stable Manifold Theorem it follows that the unstable set
\[
\mathbb W^u(p)=\{z\in\mathbb C^2|\,\Vert f^{-n}(z)-f^{-n}(p)\Vert\to 0\}
\]
is conformally equivalent to $\mathbb C$. Therefore there exists an injective parametrization $\xi_p:\mathbb C\rightarrow \mathbb W^u(p)$ and by Theorem 1 of \cite{BS2} $\mathbb W^u(p)\subset J^-$. By the definition of unstable set, the collection of all those injective parameterizations satisfies \eqref{p1}, \eqref{p2} and \eqref{p3}. This set of parameterizations induce the family $\Psi_{\mathcal S_1}$ as described above.
\end{example}
\begin{example}[Recentered Unstable Manifold]Given a single saddle fixed point $p$, we define
$$
\mathcal S_2(p)=W^u(p)\cap W^s(p).
$$
By Theorem 9.6 of \cite{BLS} $\mathcal S_2(p)$ is dense in $J^\star$. It is clear that this set is also $f$-invariant.

We consider an injective parameterization $\xi_{p}:\mathbb C\rightarrow W^u(p)$. Given $q\in\mathcal S_2(p)$, we let $\Psi_q$ be the set of the maps of the form $\psi_q(\zeta)=\xi_{p}(a_q\zeta+b_q)$ which satisfy \eqref{normalizationprop}. Taking the union of all those local family we find the family $\Psi_{\mathcal S_2(p)}$ as described above.
\end{example}

In \cite{BS8} it was proved that the condition of being quasi-expanding is independent from the choice of the set $\mathcal S$. In particular it follows that:

\begin{proposition}
\label{quasiexpsaddlepoint}
If $f$ is quasi-expanding then the families $\Psi_{\mathcal S_1}$ and $\Psi_{\mathcal S_2(p)}$ are normal families.
\end{proposition}

Consider a family $\Psi_{\mathcal{S}}$ satisfying the assumptions \eqref{p1} through \eqref{normalizationprop}, and for $r>0$ define
\begin{equation}
\label{capitalM}
M(r):=\sup_{\psi\in\Psi_\mathcal{S}}\max_{|\zeta|\leq r}G^+\circ \psi(\zeta).
\end{equation}
The function $M(r)$ is a convex, increasing function of $\log(r)$. By the normalization conditions \eqref{normalizationprop} one has $M(0)=0$, $M(1)=1$ and $M(r)>1$ when $r>1$.
For every $\psi_p\in \Psi_p$ and $\psi_{f(p)}\in\Psi_{f(p)}$, there exists a constant $\lambda_p$ such that%, if we write $L_p(\zeta)=\lambda_p\zeta$, the following diagram commutes,
%\begin{figure}[h]
%\centering
%\begin{tikzcd}
%W^s(p)  \arrow[r,"f"] & W^s(f(p))\\
%\mathbb{C}\arrow[u,"\psi_p"]\arrow [r,"L_p"]& \mathbb C\arrow["\psi_{f(p)}",u]
%\end{tikzcd}
%\end{figure}
%\par\noindent
%This diagram corresponds to the functional equation
\begin{equation}
\label{functionaleq}
f\circ\psi_p(\zeta)=\psi_{f(p)}(\lambda_p\cdot \zeta)
\end{equation}
Notice that $|\lambda_p|$ does not depend from the choice of $\psi_p\in\Psi_p$ and $\psi_{f(p)}\in\Psi_{f(p)}$.

By Theorem 1.2 of \cite{BS8} we obtain the following list of equivalent definitions of quasi-expansion.
\begin{theorem}
\label{equivalentdefquasihyp}
The following are equivalent:\begin{itemize}
\item[(i)] $\Psi_\mathcal S$ is a normal family.
\item [(ii)]$M(r)<\infty$ for all $r<\infty$.
\item[(iii)]$M(r_0)<\infty$ for some $r_0>1$.
\item[(iv)] there exists a constant $\kappa>1$ such that $|\lambda_p|\geq\kappa$ for all $p\in\mathcal S$.
\end{itemize}
\end{theorem}

\begin{remark}[Notation remark]
Given a quasi-expanding map $f$. We denote as $\widehat\Psi_\mathcal S$  the set of all normal limits of the family, and for $p \in J^\star$ we define
\[
\widehat\Psi_p:=\{\widehat\psi\in\widehat\Psi_\mathcal S|\,\widehat\psi(0)=p\}.
\]
From now on the hat-notation will be used in order to distinguish between the injective parameterizations in $\Psi_\mathcal S$ and the normal limits in $\widehat\Psi_\mathcal S$.
\end{remark}
The two following propositions follow from Proposition 1.1, Proposition 1.3 and Proposition 3.1 of \cite{BSh}.

\begin{proposition}
\label{normpropnormfam}
If $f$ is quasi-expanding then $\widehat \Psi_\mathcal S$ is compact. Every $\widehat\psi_p\in\widehat\Psi_p$ satisfies \eqref{normalizationprop}. If we take $\widehat\psi_{p_1}\in\widehat\Psi_{p_1}$ and $\widehat\psi_{p_2}\in\widehat\Psi_{p_2}$ we have that
\[
\widehat\psi_{p_1}(\mathbb C)\cap\widehat\psi_{p_2}(\mathbb C)\neq \emptyset\Rightarrow\widehat\psi_{p_1}(\mathbb C)=\widehat\psi_{p_2}(\mathbb C).
\]
\end{proposition}

Given $p\in J^\star$ and $\widehat\psi_p\in\widehat{\Psi}_p$ we let $W^u(p)=\widehat\psi_p(\mathbb C)$. Since $\widehat\psi_p$ is the limit of a sequence of maps whose image is contained in $J^-$ and since $J^-$ is closed, it follows that $W^u(p)\subset J^-$. We note that $W^u(p)$ is contained in but in general may not be equal to the unstable set of $p$.

\begin{proposition}
\label{injectiveparameter}
The set $W^u(p)$ does not depend on the choice of $\widehat\psi_p\in\widehat{\Psi}_p$. Moreover we have that
\begin{itemize}
\item[(i)] $W^u(p)$ is a nonsingular manifold and there exists an injective parametrization $\xi_p:\mathbb C\rightarrow W^u(p)\subset J^-$.
\item[(ii)] $W^u(p)\subset \mathbb W^u(p)$.
\item[(iii)] $f(W^u(p))=W^u(f(p))$.
\end{itemize}
\end{proposition}

We can conclude that the family of parameterizations $\xi_p\rightarrow W^u(p)$ for $p\in J^\star$ satisfies the properties \eqref{p1}, \eqref{p2} and \eqref{p3}.

\begin{corollary}
\label{boundedlambda}
Let $f$ be quasi-expanding. Then there exist $K\geq\kappa>1$ such that for every $p\in \mathcal S$ we have $\kappa\leq|\lambda_p|\leq K$.
\end{corollary}
\begin{proof}
Suppose on the contrary that there exists $p_n\in\mathcal S$, such that $|\lambda_{p_n}|\to\infty$. We take two sequences of functions $\psi_{p_n}\in\Psi_{p_n}$ and $\psi_{f(p_n)}\in\Psi_{f(p_n)}$, such that for every $n$ the functional equation \eqref{functionaleq} is satisfied.

By taking a subsequence of $p_n$ if necessary, by quasi-expansion and the compactness of $J^\star$ we may assume that $p_n\to q\in J^\star$, that $\psi_{p_n}\to\widehat\psi_q\in\widehat\Psi_q$ and that $\psi_{f(p_n)}\to\widehat\psi_{f(q)}\in\widehat\Psi_{f(q)}$.

By \eqref{functionaleq} we obtain that\begin{align*}
\max_{|\zeta|\leq \lambda_{p_n}}G^+\circ\psi_{f(p_n)}(\zeta)&=\max_{|\zeta|\leq 1}G^+\circ\psi_{f(p_n)}(\lambda_{p_n}\zeta)\\
&=\max_{|\zeta|\leq 1}G^+\circ f\circ\psi_{p_n}(\zeta)\\
&=d\max_{|\zeta|\leq 1}G^+\circ\psi_{p_n}(\zeta)\\
&=d.
\end{align*}
Given $R>0$ we can find a natural number $n_0$ such that $|\lambda_{p_n}|\geq R$ when $n\geq n_0$. It follows that $\max_{|\zeta|\leq R}G^+\circ\psi_{f(p_n)}(\zeta)\leq d$ for every $n\geq n_0$. It is clear that $G^+\circ\psi_{f(p_n)}\to G^+\circ\widehat\psi_{f(q)}$ locally uniformly on $\mathbb C$, therefore for every $R>0$
\[
\max_{|\zeta|\leq R} G^+\circ \widehat\psi_{f(q)}(\zeta)\leq d.
\]
In particular the subharmonic function $g(\zeta):=G^+\circ\widehat\psi_{f(q)}(\zeta)$ is bounded on $\mathbb C$. By the subharmonic version of Liouville's Theorem, $g$ is constant, and $g(\zeta)=g(0)=0$ for all $\zeta\in\mathbb C$. But this gives a contradiction, since by Proposition \ref{normpropnormfam} we have $\max_{|\zeta|\leq 1}g(\zeta)=1$.

%By the previous proposition $W^u(f(q))=\widehat\psi_{f(q)}(\mathbb C)$ is contained in $J^-$. This implies that $\widehat\psi_{f(q)}(\mathbb C)\subset\tilde K:=\{z\in \mathbb C^2|\, G^-(z)=0,\,G^+(z)\leq d\}$. Since both $G^+$ and $G^-$ have a logarithmic growth when $\Vert z\Vert\to\infty$, the set $\tilde K$ must be a bounded. {\bf HP: I don't see how this follows ffrom the logarithmic growth. Note that the logarithmic growth is for the limsup, you can still go to infinity and stay bounded. However, the fact that $\tilde{K}$ is bounded follows from the filtration}. By Liouville's theorem the components of $\widehat\psi_{f(q)}$ are constant functions. This gives a contradiction, since by Proposition \ref{normpropnormfam} we have $G^+\circ\widehat\psi_{f(q)}(0)=0$ while $\max_{|\zeta|\leq 1}G^+\circ\widehat\psi_{f(q)}(\zeta)=1$.
\end{proof}

\section{Canonical family of parameterizations}\label{section:maximal}

Throughout this section we assume that the H\'enon map $f$ is quasi-expanding. Our goal in this section is to introduce a canonical family of parameterizations of unstable manifolds through all points in $J^\star$. This family will be independent of the set $\mathcal S\subset J^\star$ and the family $\Psi_\mathcal S$.

Given a normal limit $\widehat\psi\in\widehat\Psi_\mathcal S$ we let $ord(\widehat\psi):=\min\{k\geq 1|\,\widehat\psi^{(k)}(0)\neq 0\}$. Moreover, for every $p\in J^\star$, we let
\begin{equation}
\label{tau}
\tau(p)=\sup_{\psi_p\in\widehat\Psi_p}ord(\psi_p)
\end{equation}
\begin{lemma}
\label{taubound}
Let $f$ be quasi-expanding. Then there exists an $m<\infty$ such that $\tau(p)\leq m$ for all $p\in J^\star$.
\end{lemma}
\begin{proof}
By Proposition \ref{normpropnormfam}, every $\widehat\psi\in\widehat \Psi_\mathcal S$ satisfies \eqref{normalizationprop}. Thus $\widehat\psi$ is not a constant function, and $ord(\widehat\psi)<\infty$.

Suppose that there exist a sequence $\widehat\psi_n\in\widehat\Psi_\mathcal S$ such that $ord(\widehat\psi_n)\to\infty$. By taking a subsequence if necessary we may assume that $\widehat\psi_n\to\widehat\psi_\infty\in\widehat\Psi_\mathcal S$. By uniform convergence $\widehat\psi^{(k)}_n(0)\to\widehat\psi_\infty^{(k)}(0)$, it follows that $\widehat\psi_\infty^{(k)}(0)=0$ for every $k\in \mathbb N$, which gives a contradiction.
\end{proof}

We let $\kappa>1$ as in Proposition \ref{equivalentdefquasihyp} and define $\gamma=1/\kappa$.

\begin{lemma}
\label{BresciaSuni}
Given $\widehat\psi_p\in\widehat\Psi_p$, there exist a $\widehat\psi_{f^{-1}(p)}\in\widehat\Psi_{f^{-1}(p)}$ and a constant $\lambda_{p,-1}$ with $|\lambda_{p,-1}|\leq\gamma$, such that
\[
f^{-1}\circ\widehat\psi_p(\zeta)=\widehat\psi_{f^{-1}(p)}(\lambda_{p,-1} \cdot \zeta).
\]
\end{lemma}
\begin{proof}
Given $\widehat\psi_p\in\widehat \Psi_p$, take a sequence $\psi_{p_n}\in\Psi_{p_n}$ such that $p_n\to p$ and $\psi_{p_n}\to\widehat\psi_p$, and a sequence $\psi_{f^{-1}(p_n)}\in\Psi_{f^{-1}(p_n)}$ converging to $\widehat\psi_{f^{-1}(p)}\in\widehat\Psi_{f^{-1}(p)}$.

By Proposition \ref{equivalentdefquasihyp} there exists a sequence of constants $\lambda(p_n,-1)$ satisfying $|\lambda(p_n,-1)|\leq\gamma$ for which
\[
f^{-1}\circ \psi_{p_n}(\zeta)=\psi_{f^{-1}(p_n)}(\lambda_{p_n,-1}\cdot \zeta).
\]
By converting to a subsequence if necessary, we may suppose that the sequence $\lambda_{p_n,-1}$ converges to a number $\lambda_{p,-1}$. It is clear that $|\lambda_{p,-1}|\leq \gamma$ and by the uniform convergence on compact subsets of $\psi_{p_n}$ and $\psi_{f^{-1}(p_n)}$ we obtain
\[
f^{-1}\circ\widehat\psi_p(\zeta)=\widehat\psi_{f^{-1}(p)}(\lambda_{p,-1} \cdot \zeta).
\]
\end{proof}

\begin{lemma}
\label{degbound}
Let $m$ as in Lemma \ref{taubound}. Then $\deg(\widehat\psi)\leq m+1$ for all $\widehat\psi\in\widehat\Psi_\mathcal S$.
\end{lemma}
\begin{proof}
Suppose there exists $\widehat\psi_p\in\widehat\Psi_p$ for which $\deg(\widehat\psi_p)>m+1$. Let $z\in W^u(p)$ be a regular value of $\psi_p$ so that there exist $m+2$ distinct elements $\{\zeta_1,\dots,\zeta_{m+2}\}\subset\psi^{-1}(\{z\})$.

By Lemma \ref{BresciaSuni}, there exists a sequence of maps $\widehat\psi_n\in\widehat\Psi_{f^{-n}(p)}$ and a sequence of constants $\lambda_{p,-n}$, with $|\lambda_{p,-n}|\leq \gamma^{-n}$, for which
\[
f^{-n}\circ\widehat\psi_p(\zeta)=\widehat\psi_n(\lambda_{p,-n}\cdot \zeta).
\]
We take a subsequence $n_k$ such that $\widehat\psi_{n_k}\to\widehat\psi_q\in\widehat\Psi_q$ locally uniformly on $\mathbb C$. By Proposition \ref{injectiveparameter}, $W^u(p)\subset\mathbb W^u(p)$, therefore $q_k:=\widehat\psi_{n_k}\left(\lambda_{p,-n_k}\cdot \zeta_i\right)\to q$ as $k\to\infty$.

Given $\varepsilon>0$, there exists $k_0$ such that for every $k\geq k_0$ and every $i=1,\ldots , m+2$ we have that $\lambda_{p,-n_k}\cdot\zeta_i\in\Delta_{\varepsilon}$.

Writing $\pi_1$ and $\pi_2$ for the respective projections to the $z$- and $w$-axis, we obtain
\[
\lim_{k\to\infty}\frac{1}{2\pi i}\int_{\partial \Delta_\varepsilon}\frac{(\pi_i\circ \widehat\psi_{n_k})'(\zeta)}{\pi_i\circ\widehat\psi_{n_k}(\zeta)-\pi_i(q_{k})}\,d\zeta\geq m+2.
\]
By the identity principle we can choose $\varepsilon>0$ such that $0$ is the only solution of $\widehat\psi_q(\zeta)=q$ inside $\overline{\Delta}_\varepsilon$. In particular we may assume that $|\pi_i\circ\widehat\psi_q(\zeta)-\pi_i(q)|>\delta$ on $\partial \Delta_\varepsilon$ for $i = 1, 2$.

Using uniform convergence, we bring the limit inside the integral to obtain
\[
\frac{1}{2\pi i}\int_{\partial \Delta_\varepsilon}\frac{(\pi_i\circ \widehat\psi_{q})'(\zeta)}{\pi_i\circ\widehat\psi_{q}(\zeta)-\pi_i(q)}\,d\zeta\geq m+2.
\]
It follows that $0$ is a solution of $\widehat\psi_q(\zeta) = q$ with multiplicity at least $m+2$. This implies that $ord(\widehat\psi_q)>m$, which is not possible by Lemma \ref{BresciaSuni}.
\end{proof}

\begin{corollary}
\label{polynomialconnection}
Let $\widehat\psi_p\in\widehat\Psi_p$, and let $\xi_p:\mathbb C\rightarrow W^u(p)$ be an injective parameterization of $W^u(p)$. Then there exists a polynomial $h_p$ of degree at most $m+1$ such that
\[
\widehat\psi_p(\zeta)=\xi_p\circ h_p(\zeta).
\]
\end{corollary}
\begin{proof}
The entire function $h_p=\xi^{-1}_p\circ\psi_p$ has degree at most $m+1$, and must therefore be a polynomial.
\end{proof}

Given $p\in J^\star$ we consider the family of all parameterizations $\phi_p(\zeta)=\xi_p(a\zeta+b)$ satisfying
$$
\phi_p(0) = p \; \; \mathrm{and} \; \; \max_{|\zeta| \le 1} G^+ \circ \phi_p(\zeta) = 1
$$
as in \eqref{normalizationprop}. We denote the family of such parametrizations by $\Phi_p$, and write $\Phi_{\mathcal S}=\bigcup_{p\in J^\star}\Phi_p$. The family $\Phi_\mathcal S$ could a priori depend on $\mathcal S$ and $\Psi_\mathcal S$, but it turns out that this is not the case.

\begin{theorem}
\label{bignormalfamily}
The family $\Phi_\mathcal S$ is normal and does not depend on the set $\mathcal S$ and the family $\Psi_\mathcal S$.
\end{theorem}

Theorem \ref{bignormalfamily} follows from the \textit{Proper, Bounded Area Condition}, see Theorem 3.4 of \cite{BS8} and Proposition 1.3 of \cite{BSh}. We will give a different proof, using Lemma \ref{normalfamilypoly} below, which will later be used again.

Recall from above that given a point $p\in J^\star$ and a pair of functions $\widehat\psi_p\in\widehat\Psi_p$ and $\phi_p\in\Phi_p$, there exists a polynomial $h_p$ of degree at most $m+1$ for which $\widehat\psi_p=\phi_p\circ h_p$. Let $\mathcal H_p\subset Pol(m+1)$ be the collection of these polynomials $h_p$, and define
\[
\mathcal H=\bigcup_{p\in J^\star} \mathcal H_p\subset Pol(m+1).
\]
Since $\widehat\psi_p(0)=p$, $\phi_p(0)=p$ and since $\phi_p$ is injective it follows that $h_p(0)=0$ for every $h_p\in\mathcal H$.
\begin{lemma}
\label{normalfamilypoly}
There exists a constant $M>0$ such that, given $h\in\mathcal H$, the absolute value of every coefficient of $h$ is bounded by $M$. Moreover the family $\mathcal H$ is normal and every normal limit is a non-constant polynomial of degree at most $m+1$.
\end{lemma}
\begin{proof}
Suppose that there exists a sequence $h_n\in\mathcal H$, with
$$
h_n(\zeta)=a_{m+1}(n)\zeta^{m+1}+\dots + a_1(n)\zeta,
$$
and for which
\begin{equation}
\label{greve}
\max\{|a_1(n)|,\dots,|a_{m+1}(n)|\}\to\infty.
\end{equation}
For every $n$ there exists $p_n\in J^\star$, $\widehat\psi_n\in\widehat \Psi_{p_n}$ and $\phi_n\in\Phi_{p_n}$ such that $\widehat\psi_n=\phi_n\circ h_n$.

By Proposition \ref{normpropnormfam} we know that
\[
\max_{|\zeta|\leq 1}G^+\circ\widehat\psi_n=1,
\]
thus, by the subharmonicity of $G^-\circ\widehat\psi_n$, there exist $\zeta_n$ with $|\zeta_n|=1$ for which $G^+\circ \widehat\psi_n(\zeta_n)=1$.

Note that $\tilde\psi_n(\zeta)=\widehat\psi_n(e^{i\theta}\zeta)$ is also an element of $\widehat\Psi_{p_n}$. Moreover we have that $\tilde\psi_n=\phi_n\circ \tilde h_n$, where
\[
\tilde h_n(\zeta)=a_{m+1}(n)e^{i(m+1)\theta}\zeta^{m+1}+\dots+a_1(n)e^{i\theta}\zeta.
\]
The coefficients $\tilde a_i(n)$ satisfy \eqref{greve}. By replacing $h_n$ with $\tilde h_n$ if necessary, we may suppose that each $\zeta_n$ equals $1$.

Let $0<\varepsilon<1$. Point $(3)$ of Theorem 1.2 of \cite{BS8}, implies the existence of a $\kappa<1$ such that $\max_{|\zeta|\leq \varepsilon} G^+\circ\widehat\psi_n\leq \kappa<1$ for all $n \in \mathbb N$.

We claim there exists a $\delta>0$ such that $h_n(\Delta_\varepsilon)\cap \Delta(1,\delta)=\emptyset$ for every $n\in\mathbb N$. Suppose that this is not the case. Then by taking a subsequence if necessary, we may assume that there exists $\omega_n\in h_n(\Delta_\varepsilon)$ such that $\omega_n\to 1$. By normality of $\{\widehat\psi_n\}$ we can take a subsequence if necessary so that $\widehat\psi_n\to\widehat\psi_\infty$ locally uniformly on $\mathbb C$. As a consequence, $g_n=G^+\circ\widehat\psi_n$ converges locally uniformly to $g_\infty$. But by uniform convergence, given $\omega_n\to 1$, we cannot have at the same time $g_n(\omega_n)\leq \kappa<1$ and $g_n(1)=1$, which completes the proof of our claim.

By the strong version of Montel theorem the family $\{h_n|_{\Delta_\varepsilon}\}$ is normal, contradicting \eqref{greve}, hence the coefficients of the polynomials in $\mathcal H$ are uniformly bounded. Normality of $\mathcal H$ follows immediately.

Suppose for the purpose of a contradiction that there exist a sequence $h_n\in \mathcal H$ for which the maps $h_n$ converge to a constant $c$. Since $h_n(0)=0$ it follows that $c=0$. If we write $h_n(\zeta)=a_{m+1}(n)\zeta^{m+1}+\dots +a_1(n)\zeta$ we have that $a_i(n)\to 0$ for every $i=1,\dots,m+1$ , therefore for every $\varepsilon>0$ there exists $n_0$ such that when $n\geq n_0$ we have that $|a_i(n)|\leq\varepsilon$ for every $i$. Choose $\varepsilon$ small enough in order to have $m\varepsilon<1$ and $n\geq n_0$. If $\widehat\psi_n=\phi_n\circ h_n$, then
\begin{align*}
\max_{|\zeta|\leq 1} G^+\circ\widehat\psi_n(\zeta)&=\max_{|\zeta|\leq 1} G^+\circ\phi_n\circ h_n(\zeta)\\
&=\max_{\zeta\in h_n(\Delta)}G^+\circ \phi_n(\zeta)\\
&\leq \max_{|\zeta|\leq m\varepsilon}G^+\circ\phi_n(\zeta)\\
&<1,
\end{align*}
where the last inequality follows from subharmonicity of $G^+\circ\widehat\psi_n$. This gives a contradiction with equation \eqref{normalizationprop}, completing the proof.
\end{proof}

\begin{proof}[Proof of Theorem \ref{bignormalfamily}]
Suppose on the contrary that $\Phi_\mathcal S$ is not a normal family. By point $(iii)$ of Theorem \ref{equivalentdefquasihyp} we can choose $\phi_n\in\Phi_\mathcal S$ and $\zeta_n\in\Delta_2$ such that
\[
G^+\circ\phi_n(\zeta_n)\to\infty.
\]
For every $n$ take $\widehat\psi_n\in\widehat\Psi_{\phi_n(0)}$ and let $h_n\in\mathcal H$ such that $\widehat\psi_n=\phi_n\circ h_n$.

By taking a subsequence if necessary, by the previous lemma we may assume that $h_n$ converge to the non constant polynomial $h_\infty$ locally uniformly on $\mathbb C$. Furthermore we assume that $\zeta_n\to\zeta_\infty\in\overline{\Delta_2}$.

Let $\omega_\infty$ be a solution of $h_\infty(\zeta) = \zeta_\infty$. It follows that $h_n(\omega_\infty) - \zeta_\infty \rightarrow 0$. Since $\zeta_n \rightarrow \zeta_\infty$, uniform convergence of the polynomials $h_n$ in a neighborhood of $\omega_\infty$ implies the existence of a sequence $\omega_n \rightarrow \omega_\infty$ with $h_n(\omega_n) = \zeta_n$.

%If $h_\infty(\zeta)=a_k\zeta^k+\dots+a_1\zeta$ with $a_k\neq 0$, then the product of the absolute values of the roots of $h_\infty(\zeta)-\zeta_\infty=0$ is equal to $|\zeta_\infty/a_k|$. It follows that there exists a root $\omega_\infty$ with $|\omega_\infty|\leq r_\infty:=|\zeta_\infty/a_k|^{1/k}$. \marginpar{This computation is correct but does not seem relevant. Just take any root $\omega_\infty$   and use that $h_n(\omega_\infty) - \zeta_\infty \rightarrow 0$?

%Since $h_n(\zeta)-\zeta_n\to h_\infty(\zeta)-\zeta_\infty$, there exists a sequence $\omega_n\in\mathbb C$ such that $h_n(\omega_n)=\zeta_n$ and $\omega_n\to \omega_\infty$.

Therefore $\omega_n\in\Delta_{r_\infty+1}$ for $n$ large enough, hence Theorem \ref{equivalentdefquasihyp} implies that $G^+\circ\widehat\psi_n(\omega_n)<M(r_\infty+1)$, where $M(r)$ is defined as in \eqref{capitalM} for the family $\Psi_{\mathcal S}$. Since $G^+\circ\widehat\psi_n(\omega_n)=G^+\circ\phi_n(\zeta_n)$ we have obtained a contradiction.

Let $\Psi_{\mathcal S}$ be a normal family of parametrizations as in the definition of quasi-expansion. Let $\Psi_{\mathcal S_1}$ be the family of parameterizations relative to the saddle point described in the previous section. By Proposition \ref{quasiexpsaddlepoint}, this family is normal.  %We will use the notation $\Psi_{p,i}$ for $i=0,1$ in order to distinguish the two local families at the point $p$.

Given any point $p\in J^\star$ we consider $W^u_\mathcal S(p)$ and $W^u_{\mathcal S_1}(p)$ defined respectively for the families $\Psi_{\mathcal S}$ and $\Psi_{\mathcal S_1}$. Is it clear from the definition of $\Phi_\mathcal S$ that if we can prove the equality $W^u_{\mathcal S}(p)=W^u_{\mathcal S_1}(p)$ for every $p\in J^\star$, this would imply that $\Phi_{\mathcal S}=\Phi_{\mathcal S_1}$, which finally would complete the proof of the theorem.

Let $q\in\mathcal S_1$ be a saddle point. By Proposition \ref{injectiveparameter} we know that $W^u_\mathcal S(q)\subset \mathbb W^u(q)=W_{\mathcal S_1}(q)$, where the last equality follows from the stable manifold theorem. Since $W^u_\mathcal S(p)$ and $W_{\mathcal S_1}^u(q)$ are both biholomorphic to $\mathbb C$, it follows that $W^u_{\mathcal S}(q)=W^u_{\mathcal S_1}(q)$.\\
Any $\psi\in\Psi_{q,\mathcal S_1}$ is parameterization of $W^u_\mathcal S(q)$ satisfying \eqref{normalizationprop}. It follows that $\psi\in\Phi_{\mathcal S}$.

As a consequence we obtain the two inclusion $\Psi_{\mathcal S}\subset \Phi_{\mathcal S}$ and $\Psi_{\mathcal S_1}\subset \Phi_{\mathcal S}$.\\
Since $\Phi_{\mathcal S}$ is a normal family, if we take $\mathcal S_2=J^\star$ and $\Psi_{\mathcal S_2}=\Phi_{\mathcal S}$, they satisfy the condition required in the definition of quasi-expansion. For every $p\in J^\star$ we have a third stable manifold $W^u_{\mathcal S_2}(p)$.\\
Since $\widehat{\Psi}_{p,\mathcal S}\subset\widehat\Phi_{p,\mathcal S}=\widehat\Psi_{p,\mathcal S_2}$, by Proposition \ref{injectiveparameter} it follows that $W^u_{\mathcal S}(p)=W^u_{\mathcal S_2}(p)$. In the same way we get that $W^u_{\mathcal S_1}(p)=W^u_{\mathcal S_2}(p)$ which proves the theorem.

\end{proof}
Since the family $\Phi_\mathcal S$ is normal and contains parameterizations through every point in $J^\star$, the pair $(J^\star,\Phi_\mathcal S)$ can be used in the definition of quasi-expansion. In fact, $\Phi_\mathcal S$ is the maximal family of parameterizations satisfying the required properties.\\
Since $\Phi_\mathcal S$ is independent of $\mathcal S$ we will from now on denote it by $\Psi$ or $\Psi_{J^\star}$. We will denote with $\Psi_p$, $\widehat \Psi$ and $\widehat \Psi_p$ the sets defined in section \ref{section:quasi} corresponding to the family $\Psi$. We call $\Psi$ the \emph{canonical family of parameterizations}.

\section{Stably-unbridged parametrizations}\label{section:unbridged}

Let $f$ be quasi-expanding and recall the canonical family of parameterizations $\Psi$ introduced in the previous section. If $f$ and $f^{-1}$ are both quasi-expanding (in which case $f$ is quasi-hyperbolic) we write $\widehat\Psi^u$ and $\widehat\Psi^s$ in order to distinguish the two families.

\begin{definition}
Suppose $f$ is quasi-contracting. Let $\widehat\psi\in\widehat\Psi^s$, we say that $\widehat\psi$ is \textit{stably-unbridged} if, for every $R>0$, every connected component of $\{G^-\circ\widehat\psi<R\}$ is bounded. We say that $\widehat\psi$ is \textit{stably-bridged} if it is not stably-unbridged.
\end{definition}

\begin{proposition}
Let $p\in J^\star$. Given $\widehat\psi_1,\widehat\psi_2\in\widehat \Psi_p^s$, then $\widehat\psi_1$ is stably-bridged if and only if $\widehat\psi_2$ is stably-bridged.
%If $\widehat\psi_{p,0}\in\widehat\Psi^s_p$ is unbridged (respectively bridged) then every $\widehat\psi_p\in\widehat\Psi^s_p$ is unbridged (respectively bridged).
\end{proposition}
\begin{proof}
Given $\psi_p \in \Psi^s_p$ we can write $\widehat\psi_{1} = \psi_p \circ h_{1}$ and $\widehat{\psi}_2 = \psi_p \circ h_2$ for polynomials $h_1$ and $h_2$. It follows that $\widehat\psi_1$ is stably-bridged if and only $\psi_p$ is stably-bridged, and thus if and only if $\widehat\psi_2$ is stably-bridged.
%Suppose $\psi_p\in\widehat\Psi^s_p$ is unbridged and take $\tilde\psi_p\in\Psi^s_p$. By Corollary \ref{polynomialconnection} $\psi_p=\tilde\psi_p\circ h_p$ for some polynomial $h_p\in\mathcal H$. Given $R>0$ and a connected component $\tilde U$ of $\{G^-\circ \tilde\psi_p<R\}$, every connected component of $h_p^{-1}(\tilde U)$ is contained in a connected component of $\{G^-\circ\psi_p<R\}$. Since $\psi_p$ is unbridged, $\tilde U$ must be bounded. Therefore $\tilde\psi_p$ is unbridged. Similarly if we take $\tilde{\tilde\psi}_p\in\widehat\Psi^s_p$ and $\tilde\psi_p\in\Psi^s_p$ using the fact that $\tilde{\tilde\psi}_p=\tilde \psi\circ h_p$, for some $h_p\in\mathcal H$, it follows that $\tilde{\tilde\psi}_p$ is unbridged.
\end{proof}
\begin{definition}
We say that a point $p\in J^\star$ is \textit{stably-unbridged} (respectively \textit{stably-bridged}) if every $\widehat\psi_p\in\widehat\Psi^s_p$ is stably-unbridged (respectively stably-bridged). We say that a quasi-contracting H\'enon map is \textit{stably-unbridged} if every point $p\in J^\star$ is stably-unbridged.
\end{definition}
%By the Proposition above, every point $p\in J^\star$ is either bridged or unbridged.

Recall that for $R>0$ sufficiently large it follows from the filtration properties that $J^\star \subset \Delta^2_R$. Since $G^-$ is bounded on $\Delta^2_R$, stable-unbridgedness implies that for each $p \in J^\star$ the projection onto the second coordinate on the local stable manifold $W^s_R(p)$, the connected component of $W^s(p) \cap \Delta^2_R$ containing $p$, is a branched cover of finite degree. We refer to this as the degree of $W^s_R(p)$. In fact, these degrees are uniformly bounded:

\begin{proposition}\label{prop:degreebound}
Let $f$ be stably-unbridged and $R>0$ as above. Then the degrees of the stable manifolds $W_R^s(p)$ are uniformly bounded.
\end{proposition}
\begin{proof}
The proof is identical to that of Lemma 5.1 in \cite{LP2}:

Suppose that there is no bound on the degrees. Then there exist a sequence $(p_n)$ for which the degrees of $W_R^s(p_n)$ converge to infinity. By passing to a subsequence we may assume that $p_n \rightarrow p \in J^\star$. For $R_1>R$ it then follows that the degree of $W_{R_1}^s(p)$ is infinite, giving a contradiction.
\end{proof}

The aim of this section is to prove the following theorem:
\begin{theorem}
\label{boundedlevset}
Let $f$ be quasi-hyperbolic and substantially dissipative. Then $f$ is stably-unbridged.
\end{theorem}

We define
\[
O(p)=\overline{\bigcup_{n\in \mathbb N}f^n(p)}
\]
\begin{lemma}
\label{bridgedorbit}
If $p\in J^\star$ is stably-bridged then every point in $O(p)$ is stably-bridged.
\end{lemma}
\begin{proof}
If $p$ is stably-bridged, given $\psi_p\in\Psi^s_p$, then there exists $R>0$ and a component $U_0$ of $\{G^-\circ\psi_p<R\}$ which is unbounded.

Consider first $q=f^n(p)$ and take $\psi_q\in\Psi^s_q$. Let $\lambda_{p,n} \in \mathbb C$ be such that $f^n\circ\psi_p(\zeta)=\psi_q(\lambda_{p,n}\cdot \zeta)$, so that
\[
G^-\circ\psi_q(\zeta)=\frac{1}{d^n}G^-\circ\psi_p\left(\frac{\zeta}{\lambda_{p,n}}\right).
\]
It is clear that if $\zeta\in\lambda_{p,n} \cdot U_0$  then $G^-\circ\psi_q(\zeta)<R/d^n$. Since $\lambda_{p,n}\cdot U_0$ is connected and unbounded, the point $q$ is stably-bridged.

Consider now $q\in O(p)\setminus \bigcup_n f^n(p)$ and take $n_k$ such that $f^{n_k}(p)\to q$. For every $k$ choose $\psi_k\in\Psi_{f^{n_k}(p)}$. Since $f$ is quasi-hyperbolic, by taking a subsequence of $n_k$ if necessary, we may suppose that the sequence $\psi_k$ converges locally uniformly on $\mathbb C$ to $\widehat\psi_q\in\widehat\Psi_q$.

As above, let $\lambda_{p,n_k}$ be the sequence of constants such that $G^-\circ\psi_k(\zeta)=\frac{1}{d^n}G^-\circ\psi_p\left(\frac{\zeta}{\lambda_{p,n_k}}\right)$. It is clear that $G^-\circ \psi_k(\zeta)<R/d^n\leq R$ for $\zeta\in U_k:=\lambda_{p,n_k}\cdot U_0$, moreover, by quasi-hyperbolicity, there exists $\kappa<1$ such that $|\lambda_{p,n_k}|\leq \kappa^{n_k}$.

Let $g_k=G^-\circ \psi_k$ and $g_\infty=G^-\circ\widehat\psi_q$. The sequence $g_k$ converges locally uniformly on $\mathbb C$ to $g_\infty$. Let $m\in\mathbb N$, then we can choose $k_m$ such that if $k\geq k_m$ we have
\[
\sup_{|\zeta|\leq m}|g_k(\zeta)-g_\infty(\zeta)|<R.
\]
Let $V_m=U_{k_m}\cap\Delta_m$. If $\zeta\in V_m$, then $g_\infty(\zeta)<2R$. Since $g_\infty(0)=0$ and $g_\infty$ is continuous, then there exists $\delta$ such that $g_\infty(\zeta)<2R$ for $\zeta\in\Delta_\delta$.

Since $U_k=\lambda_{p,n_k}\cdot U_0$ and $\lambda_{p,n_k}\to 0$, there exists $M$ such that $V_m\cap \Delta_\delta\neq\emptyset$ for $m\geq M$, and since for every $k$ the set $U_k$ is unbounded, we have $diam(V_m)\to\infty$ as $m\to\infty$. Hence the set
\[
V=\Delta_\delta\cup\bigcup_{m\geq M}V_m
\]
is unbounded and connected. If $\zeta\in V$ then $g_\infty(\zeta)<2R$. Therefore $q$ is stably-bridged.
\end{proof}

Given $p\in J^\star$ we define the maximal orders $\tau^s(p)$ and $\tau^u(p)$ of the respective families $\widehat\Psi_p^s$ and $\widehat\Psi_p^u$ as in \eqref{tau}. We let
\[
J^\star_{m^s,m^u}:=\left\{p\in J^\star|\,\tau^s(p)=m^s,\,\tau^u(p)=m^u\right\}.
\]
\begin{lemma}
\label{tauincrease}
Let $p\in J^\star$. For every $q\in O(p)$, we have $\tau^{s/u}(q)\geq \tau^{s/u}(p)$.
\end{lemma}
\begin{proof}
We will give the proof for the unstable family, the proof is similar for the stable family.

Suppose that $\tau^u(p)=m_p$ and let $\widehat\psi_p\in\widehat\Psi^u_p$ such that $ord(\widehat\psi_p)=m_p$. Take a sequence $p_n\in J^\star$ and $\psi_{p_n}\in\Psi^u_{p_n}$ satisfying $\psi_{p_n}\to \widehat\psi_p$. We also choose a sequence $\psi_{f(p_n)}\in\Psi^u_{f(p_n)}$, and constants  $\lambda_{p_n}$ for which $f\circ\psi_{p_n}(\zeta)=\psi_{f(p_n)}(\lambda_{p_n} \cdot \zeta)$. By Corollary \ref{boundedlambda} there exists $K>1$ independent of $n$ for which $|\lambda_{p_n}|\leq K$.

By taking a subsequence if necessary, we may assume that $\psi_{f(p_n)}\to\widehat\psi_{f(p)}\in\widehat\Psi^u_{f(p)}$, locally uniformly on $\mathbb C$, and that $\lambda_{p_n}\to\lambda_p$. We obtain
\[
f\circ\widehat\psi_p(\zeta)=\widehat\psi_{f(p)}(\lambda_p \cdot \zeta).
\]
It follows that $ord(\widehat\psi_{f(p)})\geq ord(\widehat\psi_p)$ and thus $\tau^u(f(p))\geq \tau^u(p)$.

Let $q\in O(p)$ and take a sequence $n_k$ such that $f^{n_k}(p)\to q$. For every $k$, $\tau^u(f^{n_k}(p))\geq\tau^u(p)$, therefore we can take $\widehat\psi_k\in\widehat\Psi^u_{f^{n_k}(p)}$ with $ord(\widehat\psi_k)\geq m_p$. By taking a subsequence if necessary, we may suppose that $\widehat\psi_k\to\widehat\psi_\infty\in\widehat\Psi^u_q$. It follows that $ord(\widehat\psi_\infty)\geq m_p$, therefore $\tau^u(q)\geq\tau^u(p)$.
\end{proof}
\begin{remark}
Using the inverse function $f^{-1}$ it follows that $\tau^{s/u}(p)\geq\tau^{s/u}(f(p))$. As a consequence the sets $J^\star_{m^s,m^u}$ are invariant with respect to $f$ and $f^{-1}$.
\end{remark}
\begin{lemma}
\label{taustrictincrease}
Let $p\in J^\star$ and let $\psi_n\in\Psi^{u}_{f^n(p)}$ (respectively $\Psi^s_{f^n(p)}$). If there exists a subsequence $n_k$ such that $\Vert\psi'_{n_k}(0)\Vert\to 0$, then for every $q$ in the limit set of the sequence $f^{n_k}(p)$ we have $\tau^{u}(q)>\tau^{u}(p)$ (respectively $\tau^{s}(q)>\tau^{s}(p)$).
\end{lemma}
\begin{proof}
We again prove the statement only for the unstable family.

Given $n\in\mathbb N$, by the previous lemma, $\tau^u(f^n(p))\geq \tau^u(p)=:m_p$, hence there exists $\widehat\psi_n\in\widehat \Psi_{f^n(p)}$ such that $ord(\widehat\psi_n)\geq m_p$. For every $n$ there exists $h_n\in \mathcal H$ with $\widehat\psi_n=\psi_n\circ h_n$. Since $\psi_n'(0)\neq 0 $ and the degree of $h_n$ is uniformly bounded by a certain $M>0$, the polynomial $h_n$ can be written as
\[
h_n(\zeta)=a_M(n)\zeta^M+\dots+a_{m_p}(n)\zeta^{m_p}.
\]
Since $h_n^{(j)}(0) = 0$ for all $j < m_p$ we have
\begin{align*}
\widehat\psi_n^{(m_p)}(0)&=\psi_n'(0)h_n^{(m_p)}(0)\\
&=\psi_n'(0)a_{m_p}(n)m_p!
\end{align*}
Since the constants $a_{m_p}$ have bounded norms by Lemma \ref{normalfamilypoly}, it follows that if $\psi_{n_k}'(0)\to 0$ then $\widehat \psi_n^{(m_p)}(0) \rightarrow 0$. Thus for every $q$ in the limit set of $f^{n_k}(p)$ we have $\tau^u(q) > m_p =  \tau^u(p)$.
\end{proof}

Given $p\in J^\star$, we define the unstable and the stable directions $E^{u/s}_p$ as the tangent spaces $T_pW^{u/s}(p)$. Given nonzero vectors $v^s\in E^s_p$ and $v^u\in E^u_p$ we define
\[
D(p):=\frac{|\det(v^s|v^u)|}{\Vert v^s\Vert\,\Vert v^u\Vert}.
\]
It is clear that $0\leq D(p)\leq 1$, that the value of $D$ at the point $p$ does not depend from the choice of $v^s$ and $v^u$, and that $D(p)=0$ if and only if $E^s_p=E^u_p$.

\begin{lemma}
Let $p\in J^\star_{m^s,m^u}$. Suppose that there exists a constant $\gamma>0$ such that $D(q)\geq \gamma$ for every $q\in O(p)$. Suppose also that $O(p)\subset J^\star_{m^s,m^u}$. Then there exists a constant $C_p>0$ such that for every $v_p\in E^s_p$ we have that
\[
\Vert df^n_pv_p\Vert\leq C_p |\delta|^n\Vert v_p\Vert
\]
where $\delta=\det(df_p)$.
\end{lemma}
\begin{proof}
It is sufficient to prove the statement for unit vectors $v_p$. Choose unit vectors in $E^{s/u}_p$, which we will denote by $v^{s/u}_0$. Then for every $n\in \mathbb N$ the unit vectors $v_n^{s/u}=\frac{df^nv^{s/u}_0}{\Vert df^nv^{s/u}_0\Vert}$ lie in $E^{s/u}_{f^n(p)}$.

We can choose $U_n\in SL(2,\mathbb C)$ such that $U_n(v^s_n)=D(f^n(p)) e_1$ and $U_n(v^u_n)=D(f^n(p)) e_2$, where $e_1=(1,0)$ and $e_2=(0,1)$.

The function $G_n:=U_n\, df^n_p\,U_0^{-1}$ is a linear map from $\mathbb C^2$ to itself with
\begin{align*}
G_ne_1&=\frac{D(f^n(p))}{D(p)}\Vert df^n_p v_0^s\Vert e_1=c_{1,n}(p)e_1, \; \; \mathrm{and}\\
G_ne_2&=\frac{D(f^n(p))}{D(p)}\Vert df^n_p v_0^u\Vert e_2=c_{2,n}(p)e_2.
\end{align*}
Thus $G_n$ is a diagonal matrix with determinant $\delta^n$, hence $c_{1,n}(p)c_{2,n}(p)=\delta^n$.

Take $\psi_p\in\Psi^u_p$ and for every $n$ take $\psi_n\in\Psi^u_{f^n(p)}$. If there exists a subsequence $n_k$ for which $\psi_{n_k}'(0)\to 0$, then by the previous lemma there exists $q\in O(p)$ with $\tau^u(q)>\tau^s(p)$ which contradicts our hypothesis $O(p) \subset J^\star_{m^s,m^u}$. Therefore there exist a constant $k>0$ such that $\Vert\psi_n'(0)\Vert\geq k$ for every $n\in \mathbb N$.

Since $f$ is quasi-expanding, there exists $\lambda_{p,n}$, with $|\lambda_{p,n}|>\kappa^n$ for some $\kappa>1$, such that $f^n\circ\psi_p(\zeta)=\psi_n(\lambda_{p,n}\zeta)$. We obtain that
\begin{align*}
\Vert df^n_p\psi'_p(0)\Vert&= \Vert\lambda_{p,n}\psi'_n(0)\Vert\\
&\geq k\kappa^n.
\end{align*}
The vector $v^u_0$ is a constant multiple of $\psi'_p(0)$, therefore $\Vert df_p^n v_0^u\Vert\to \infty$ as $n\to\infty$. Our hypothesis that $D(q) \ge \gamma$ for every $q \in O(p)$ implies that $|c_{2,n}(p)|\geq \Vert df_p^nv_0^u\Vert/\gamma$, which gives $c_{2,n}(p)\to \infty$. Since $c_{1,n}(p) c_{2,n}(p) = \delta^n$ we obtain $|c_{1,n}(p)|\leq|\delta|^n$ for $n$ sufficiently large. We conclude that
\begin{align*}
\Vert df_p^n v^s_0\Vert&=c_{1,n}(p)\frac{D(p)}{D(f^n(p))}\\
&\leq C_p|\delta|^n,
\end{align*}
which completes the proof.
\end{proof}

Recall that a subharmonic function $g:\mathbb C\rightarrow \mathbb R$ is said to have order of growth $r$ if $g(z)=O(|z|^r)$ as $z\to\infty$. We will use the following classical result of Wiman \cite{Wi}.

\begin{theorem}[Wiman]
Let $g$ be a non-constant subharmonic function with order of growth strictly less than $1/2$. Then all connected components of $\{g\leq R\}$ are bounded for every $R\geq 0$.
\end{theorem}

Wiman's Theorem has recently been used for the study of substantially dissipative H\'enon maps in \cite{DL} and \cite{LP} in similar ways as we use it here, namely to show that certain stable manifolds are unbridged.

\begin{lemma}
\label{boundedangles}
Let $p\in J^\star_{m^s,m^u}$, and assume that that $O(p)\subset J^\star_{m^s,m^u}$. Further suppose that there exists a constant $\gamma>0$ such that $D(q)\geq \gamma$ for every $q\in O(p)$. Then $p$ is stably-unbridged.
\end{lemma}
\begin{proof}
Take $\psi_p\in\Psi^s_p$ and for every $n$ take $\psi_n\in\Psi^s_{f^n(p)}$. By the previous lemma there exists a constant $C_p>0$ such that
\[
\Vert df^n_p\psi_p'(0)\Vert\leq C_p|\delta|^n,
\]
where $\delta=\det(df_p)$. As in the proof of the previous lemma there exists $k>0$ such that for every $n$ we have $\Vert \psi'_n(0)\Vert\geq k$. For every $n$ there exits a constant $\lambda_{p,n}$ such that $f^n\circ \psi_p(\zeta)=\psi_n(\lambda_{p,n}\zeta)$, and with
\begin{align*}
|\lambda_{p,n}|&=\frac{\Vert df^n_p\psi_p'(0)\Vert}{\Vert \psi_n'(0)\Vert}\\
&\leq C'_p|\delta|^n.
\end{align*}
Let $g_n=G^-\circ\psi_n$ and $g_p=G^-\circ\psi_p$ so that
\[
g_p(\zeta)=d^ng_n(\lambda_{p,n}\cdot \zeta).
\]
Let $n\in\mathbb N$ and take $\frac{|\delta|^{1-n}}{C_p'}\leq|\zeta|\leq \frac{|\delta|^{-n}}{C_p'}$. Then by \eqref{normalizationprop} we obtain
\begin{align*}
g_p(\zeta)&=d^ng_n(\lambda_{p,n}\cdot \zeta)\\
&\leq d^{1-\frac{\log(C'_p)}{\log|\delta|}}\,d^{-\frac{\log|\zeta|}{\log|\delta|}}\\
&\leq M|\zeta|^{-\frac{\log|d|}{\log|\delta|}}.
\end{align*}
Since the final estimate does not depend on $n$, it hold for all $\zeta \in \mathbb C$. This implies that $g_p$ has order of growth $-\frac{\log|\zeta|}{\log|\delta|}$ and by substantial dissipativity we get that $-\frac{\log|d|}{\log|\delta|}<\frac{1}{2}$.

By Wiman's Theorem it follows that every connected component of $\{g_p=G^-\circ\psi_p<R\}$ is bounded for every $R>0$, and thus the point $p$ is stably-unbridged.
\end{proof}

The following Lemma follows from Proposition 4.2 of \cite{BSh}.

\begin{lemma}
\label{tangency}
Suppose that $p\in J^\star_{m^s,m^u}$ is a point of tangency, i.e. $E^s_p=E^u_p$. Then the forward limit set $\omega(p)$ is contained in
$$
\bigcup_{p\geq m^s, q>m^u} J^\star_{p,q}.
$$
\end{lemma}

Now we are ready to prove Theorem \ref{boundedlevset}

\begin{proof}[Proof of Theorem \ref{boundedlevset}]
By Lemma \ref{taubound}, there exists a bound on both $\tau^s$ and $\tau^u$. Suppose for the purpose of a contradiction that there exists a stably-bridged point $p\in J^\star$, then there exists a maximal pair $(M^s,M^u)$ such that $J^\star_{M^s,M^u}$ contains a stably-bridged point. Here maximal means that there is no pair $(m^s,m^u)\neq(M^s,M^u)$ with $m^s\geq M^s$ and $m^u\geq M^u$ for which $J^\star_{m^s,m^u}$ contains a stably-bridged point.

A stably-bridged point $p\in J^\star_{M^s,M^u}$ cannot be a point of tangency. Indeed, all the points in $O(p)$ are stably-bridged by Lemma \ref{bridgedorbit}, and by Lemma \ref{tangency} the forward limit set $\omega(p)$ is contained in
$J^\star_{m^s,m^u}$ with $m^s\geq M^s$ and $m^u>M^u$. By the maximality of the couple $(M^s,M^u)$ this situation therefore cannot happen.

By Lemmas \ref{bridgedorbit} and \ref{tauincrease} every point $q\in O(p)$ is stably-bridged and lies in some $J^\star_{m^s,m^u}$ with $m^s\geq M^s$ and $m^u\geq M^u$. By the maximality of the pair $(M^s,M^u)$ it follows that $O(p)\subset J^\star_{M^s,M^u}$. Since stably-bridged points in $J^\star_{M^s,M^u}$ cannot be tangency points, it follows that $E^s_q \neq E^u_q$ for every point $q\in O(p)$.

For every $n \in \mathbb N$ let $\psi_n^s\in \Psi^s_{f^n(p)}$ and $\psi_n^u\in\Psi^u_{f^n(p)}$. By Lemma \ref{taustrictincrease} it follows that there cannot be a subsequence $n_k$ for which $(\psi^s_{n_k})'(0)$ or $(\psi^u_{n_k})'(0)$ converges to zero. It follows that stable and unstable directions are continuous on $O(p)$. In particular for every sequence $q_n\in O(p)$ with $q_n\to q_\infty$ one has $D(q_n)\to D(q_\infty)$. Since there are no tangencies in $O(p)$ there must be a $\gamma>0$ such that $D(q)\geq \gamma$ for every $q\in O(p)$.

But then Lemma \ref{boundedangles} implies that $p$ is stably-unbridged, which gives a contradiction. Therefore every point $p\in J^\star$ is stably-unbridged.
\end{proof}

\section{Proof of the main theorem} \label{section:proof}

In this section we will prove that
\begin{theorem}
\label{maintheorem2}
Let $f$ be a stably-unbridged quasi-contracting H\'enon map. Then $J=J^*$.
\end{theorem}
Since stably-unbridgness is defined only for quasi-contracting H\'enon maps, from now on we will drop the second term.

Given such a map, we know from Proposition \ref{injectiveparameter} that there exists a stable manifold through every $p\in J^*$. Recall that a priori it is not clear that the unstable set of a given point $q \in J^\star$ is a manifold. Following \cite{BLS}, there exists however a dense subset $\mathcal R\subset J^\star$, the set of regular points, such that through every point $q\in\mathcal R$ passes an unstable manifold $W^u(q)$.

\begin{lemma}
\label{intersectionstableunstable}
Let $f$ be a quasi-contracting map. Suppose that, given $p\in J^\star$ and $q\in\mathcal R$, the manifolds $W^s(p)$ and $W^u(q)$ intersect transversally in a point $x$. Then $x\in J^\star$.
\end{lemma}
\begin{proof}
By Theorem \ref{closuresaddle}, there exist sequences of saddle points $p_n$ converging to $p$. For $n\in \mathbb N$ let $\psi_{p_n}^s\in \Psi_{p_n}^s$. By quasi-contraction we can take a subsequence if necessary so that $\psi_{p_n}^s\to\widehat\psi_p^s\in\widehat\Psi_p^s$ locally uniformly on $\mathbb C$.

By Proposition \ref{injectiveparameter}, $W^s(p)=\widehat\psi^s_p(\mathbb C)$. By the transversality of the intersection between $W^s(p)$ and $W^u(q)$ at $x$ it follows that $W^s(p_n)$ and $W^u(q)$ intersect at points $x_n$, with $x_n \rightarrow x$. By Theorem 9.9 of \cite{BLS} $x_n\in J^\star$, hence $x\in J^\star$.
\end{proof}

Let $\mu^-$ be the $(1,1)$-current defined as in \eqref{currents}. For a submanifold $M$ we write  $\mu^-_{|M}$ for the measure on $M$ induced by the positive distribution
\[
\mu^-_{|M}(\varphi)=\int_MG^-dd^c\varphi,\qquad \varphi\in C^\infty_0(M).
\]

We recall the following from \cite{BS2}.

\begin{theorem}
Let $f$ be a H\'enon map and suppose that $M$ is a compact subset of $W^s(p)$ with $p\in J^\star$. Suppose further that $\mu^-_{|W^s(p)}(\partial M)=0$. Then
\[
\lim_{n\to\infty}d^{-n}f^{-n}_*([M])=\mu^-_{|W^s(p)}(M)\mu^+,
\]
where $d=\deg(f)$, $f^{-n}_*([M])$ denotes the current of integration over $f^n(M)$. The convergence holds in the sense of currents.
\end{theorem}

\begin{proposition}
\label{vanishingprop}
Let $f$ be a stably-unbridged H\'enon map. Suppose that $J^\star\neq J$ and that $q\in J\setminus J^\star$ lies in $W^s(p)$ for some $p\in J^\star$. Let $\widehat\psi_p\in\widehat\Psi_p^s$ and let $\omega$ such that $q=\widehat\psi_p(\omega)$. Then $G^-\circ\widehat\psi_p$ vanishes on a open neighborhood of $\omega$.
\end{proposition}
\begin{proof}
Let $\psi_p\in \Psi_p^s$ be an injective parametrization of $W^s(p)$ and let $g_p=G^-\circ\psi_p$.

Let us first assume that $g_p$ is not harmonic on any disk $\Delta_r(\omega)$, from which it follows that $0<\mu^-_{|W^s(p)}(M_r)<\infty$, where $M_r=\psi_p(\Delta_r(\omega))$.

Fix $r_0<\infty$ and let $E_n:=\{0<r\leq r_0 \mid \,\mu^-_{|W^s(p)}(\partial M_r)\geq 1/n\}$. For every $n$ the set $E_n$ is finite, therefore $E_\infty=\bigcup_{n\geq 0}E_n$ is at most a countable set. It follows that there exists a sequence $r_n\to 0$ such that, for every $n$
\[
\mu^-_{|W^s(p)}(M_{r_n})>0, \; \; \mathrm{and} \; \; \mu^-_{|W^s(p)}(\partial M_{r_n})=0.
\]
By the previous theorem we have
\[
\lim_{k\to\infty}d^{-k}f^{-k}_*([M_{r_n}])=\mu^-_{|W^s}(M_{r_n})\mu^+.
\]
Following the steps of the proof of Lemma 9.1 in \cite{BLS}, we find that, given a Pesin box P  there exists a natural number $N_n$ and $x_n\in P$ such that $f^{-N_n}(M_{r_n})$ intersects $W^u_r(x_n)$ transversally at a point $y_n$. See chapter 4 of \cite{BLS} for more details on Pesin boxes.

It is clear that $y_n\in W^s(f^{-N_n}(p))\cap W^u(x_n)$ and since the intersection is transverse, by Theorem \ref{intersectionstableunstable}, it follows that $y_n\in J^\star$.
The point $q_n=f^{N_n}(y_n)$ belongs to $M_{r_n}\cap J^\star$ and $q_n\to q$. By the compactness of $J^\star$ it follows that $q\in J^\star$, which contradicts our hypothesis.

Therefore there exists an $r>0$ for which the function $g_p$ is harmonic on $\Delta_r(\omega)$. By the maximum principle $g_p$ vanishes on $\Delta_r(\omega)$. Given $\widehat\psi_p\in\widehat \Psi_p^s$, there exists a polynomial $h_p\in\mathcal H$ such that $\widehat\psi_p=\psi_p\circ h_p$. If $U$ is an open neighborhood of $\omega$ such that $G^-\circ\psi_p$ vanishes on $U$, then $G^-\circ\widehat\psi_p$ vanishes on $h_p^{-1}(U)$, an open set containing all the points $\omega$ for which $\widehat\psi_p(\omega)=q$, completing the proof.
\end{proof}

Let $R>0$ big enough for which we have the filtration $\Delta_R^2,V^+, V^-$ described in section \ref{section:quasi}.
\begin{lemma}
\label{lemma:5.5}
Let $f$ be a stably unbridged H\'enon map. Then
$$
J^+ \cap \Delta^2_R = \bigcup_{p \in J^\star} W^s_R(p).
$$
\end{lemma}
\begin{proof}
Since each stable manifold is contained in $J^+$ the inclusion ``$\supset$'' is immediate.

Let $r<R$ for which the filtration is still valid.
Let $p \in J^\star$ and consider the intersection $W^s(p) \cap \Delta^2_r$.  Any connected component $U$ of $W^s(p) \cap \Delta^2_r$ is a properly embedded holomorphic disk, with boundary contained in $|y| = r$. We can regard $U$ as a branched cover over the vertical disk $\{|y| < r\}$ of some finite degree $k$. It follows that $V_n = f^n(U) \cap \Delta^2_r$ is a finite union of properly embedded holomorphic disks, for which the sum of the degrees equals $d^n \cdot k$. Writing $U_n = f^{-n} V_n$ we obtain a nested sequence $U \supset U_1 \supset \cdots$, whose intersection is non-empty. Thus $U$ intersects $K^- = J^-$, and hence also $J$. This intersection $U \cap J$ is relatively compact in $J$, thus by Proposition \ref{vanishingprop} $U$ also intersects $J^\star$, and thus $U = W^s_r(p_0)$ for some point $p_0 \in J^\star$.

Let $q \in J^+ \cap \Delta^2_{R}$ and let $r<R$ as above be such that $q\in \Delta^2_r$ as well. Recall that for a saddle point $p$ the stable manifold $W^s(p)$ is dense in $J^+$. In particular $W^s(p)$ accumulates at the point $q$. Let $q_n \in W^s(p)\cap \Delta^2_{r}$ be a sequence converging to $q$. For each $n$ let $p_n\in J^\star$ be such that $q_n \in W^s_{r}(p_n)$. By restricting to a subsequence we may assume that $p_n \rightarrow p_\infty \in J^\star \cap \Delta^2_r$. By Proposition \ref{prop:degreebound} the degrees of the branched coverings are uniformly bounded, from which it follows that $q \in W^s_R(p_\infty)$, which completes the proof.
\end{proof}

\begin{remark}\label{dujardin}
In \cite{D} the equality $J = J^\star$ is proved under the assumption that the current $\mu^+|_{\Delta^2_R}$ has \emph{no degree growth}. By the above lemma plus Proposition \ref{prop:degreebound} one can conclude that the no degree growth condition is satisfied for stably unbridged H\'enon maps. In what follows we give a self-contained proof of $J = J^\star$, following ideas from \cite{F}.
\end{remark}

\begin{lemma}
Let $f$ be a stably-unbridged H\'enon map, and suppose further that $J\neq J^\star$. Then there exists $\psi\in\Psi^s$  and a open set $U_0\subset \mathbb C$ with $0\in\partial U_0$ such that
\[
G^-\circ \psi(\zeta)=0,\qquad \forall \zeta\in U_0.
\]
\end{lemma}
\begin{proof}
Let $q\in J\setminus J^\star$. By the previous lemma there exists $p\in J^\star$ such that $q\in W^s(p)$. Take $\psi_p\in\Psi_p^s$ and let $\omega\in \mathbb C$ such that $q=\psi_p(\omega)$. By Proposition \ref{vanishingprop} we know that $g_p=G^-\circ\psi_p$ vanishes on some open neighborhood containing $\omega$.

Let $U$ be the connected component  of $\{g_p=0\}$ containing $\omega$. Then $U\neq\mathbb C$, and for every $\zeta_0\in\partial U$ we have $\tilde p:=\psi_p(\zeta_0)\in J^\star$. The statement of the lemma is satisfied for $\psi_{\tilde p}\in\Psi_{\tilde p}^s$.
\end{proof}

Now we are ready to prove that $J = J^\star$ for stably-unbridged quasi-contracting H\'enon maps.

\begin{proof}[Proof of Theorem \ref{maintheorem2}]
Suppose on the contrary that $J\neq J^\star$. By the Lemma above there exists $\psi_0\in\Psi^s$ such that $G^-\circ \psi_0$ vanishes in an open connected set $U_0$ with $0\in\partial U_0$.

Let $p_0=\psi_0(0)$ and let $p_n=f^{-n}(p_0)$. For every $n$, choose $\psi_n\in\Psi_{p_n}^s$ and let $g_n=G^-\circ \psi_n$. By quasi-contraction, there exist $\kappa>1$ and $\lambda_{p,-n}$ with $|\lambda_{p,-n}|\geq\kappa^n$ such that
\[
\psi_n(\lambda_{p,-n}\zeta)=f^{-n}\circ\psi_0(\zeta).
\]
By the properties of the Green's function $G^-$ we get that
\[
g_n(z)=d^ng_0\left(\frac{z}{\lambda_{p,-n}}\right).
\]
It follows that $g_n=0$ on $U_n=\lambda_{p,-n} \cdot U_0$.

By quasi-contraction, there exists a subsequence $n_k$ such that $\psi_{n_k}\to\widehat\psi_{\infty}\in\widehat\Psi^s$ locally uniformly on $\mathbb C$. Moreover we also have that $g_{n_k}\to g_\infty=G^-\circ \widehat\psi_\infty$ locally uniformly on $\mathbb C$. Let fix $\varepsilon>0$, then for every $m\in \mathbb N$ there exists $k_m$ for which
\begin{equation}
\label{dioboia}
\sup_{\zeta\in\Delta_m}\Vert g_{n_k}(\zeta)-g_{\infty}(\zeta)\Vert\leq \varepsilon,\qquad \forall k\geq k_m.
\end{equation}
We define
\[
V_m:=U_{n_{k_m}}\cap \Delta_m
\]
and write $\tilde \psi_m=\psi_{n_{k_m}}$ and $\tilde g_m=g_{n_{k_m}}$.

Since $\tilde g_m=0$ on $V_m$, then by \eqref{dioboia} it follows that $g_{\infty}\leq \varepsilon$ on $V_m$. Moreover $\mathrm{diam}(V_m)\to \infty$ as $m\to\infty$.

Since $g_\infty(0)=0$ and $g_\infty$ is continuous, there exists a disk $\Delta_\delta$, centered at the origin, such that $g_\infty\leq \varepsilon$ on $\Delta_\delta$. Now since $0\in\partial V_m$ for every $m$, it follows that the set
\[
V=\Delta_\delta\cup\bigcup_{m=1}^{\infty}V_m
\]
is unbounded and connected, and hence it is contained in an unbounded connected component of $\{g_\infty\leq \varepsilon\}$. But that contradicts the hypothesis that $f$ is stably-unbridged, which completes the proof.
\end{proof}
By Theorem \ref{boundedlevset}, if $f$ is substantially dissipative and quasi-hyperbolic, then $f$ is stably-unbridged. It is shown in \cite{LP2} that if $f$ is substantially dissipative and $f$ admits a dominated splitting over $J^*$ then it is stably-unbridged. As a consequence we obtain the two following corollaries
\begin{corollary}
Let $f$ be a substantial dissipative H\'enon map, and assume that $f$ is quasi-hyperbolic. Then $J = J^\star$.
\end{corollary}
\begin{corollary}
Let $f$ be a substantial dissipative H\'enon map, and assume that $f$ admits a dominated splitting on $J^\star$. Then $J = J^\star$.
\end{corollary}
\printbibliography
\end{document}